\documentclass[11pt]{article}
\usepackage{amsmath}
\usepackage{amsthm}
\usepackage{graphicx}
\usepackage{bbm,amssymb}
\usepackage[hyperfootnotes=false,colorlinks]{hyperref}
\usepackage{float}

\def\div{\mathrm{div} \,}
\def\setcolon{{\hspace{1ex}:\hspace{1ex}}}

\def\Grid{\mathcal{G}}
\def\Early{\mathcal{E}}
\def\Late{\mathcal{L}}
\def\R{\mathbb R}
\def\Z{\mathbb Z}
\def\PP{\mathbb P}
\def\EE{\mathbb E}

\def\one{\mathbf{1}}
\def\B{\mathbf{B}} 
\def\O{\Omega}
\def\disp{\displaystyle}

\newcommand{\floor}[1]{\lfloor {#1} \rfloor}
\newcommand{\ceil}[1]{\lceil {#1} \rceil}

\newcommand{\de}{\partial}
\newcommand{\old}[1]{}
\newcommand{\arxiv}[1]{{\tt \href{http://arxiv.org/abs/#1}{arXiv:#1}}}

\newtheorem{theorem}{Theorem}

\newtheorem{lemma}[theorem]{Lemma}
\newtheorem{corollary}[theorem]{Corollary}

\theoremstyle{remark}

\theoremstyle{definition}

\title{Internal DLA in Higher Dimensions}
\author{David Jerison \and Lionel Levine\footnote{Supported by an NSF Postdoctoral Research Fellowship.} \and Scott Sheffield\footnote{Partially supported by NSF grant
DMS-0645585.}}
\date{January 31, 2011}

\begin{document}

\maketitle

\begin{abstract}
Let $A(t)$ denote the cluster produced by internal diffusion limited aggregation (internal DLA) with $t$ particles in dimension $d \geq 3$.  We show that $A(t)$ is approximately spherical, up to an $O(\sqrt{\log t})$ error.

\end{abstract}

In the process known as internal diffusion limited aggregation (internal DLA)
one constructs for each integer
time $t \geq 0$ an {\bf occupied set} $A(t) \subset \mathbb Z^d$ as follows:
begin with $A(0) = \emptyset$ and $A(1) = \{0\}$.  Then, for each integer $t > 1$,
form $A(t+1)$ by adding to $A(t)$ the first point at which a simple random walk from the origin
hits $\Z^d \setminus A(t)$.
Let $B_r \subset \R^d$ denote the ball of radius $r$ centered at $0$, and write $\B_r := B_r \cap \Z^d$.  Let $\omega_d$ be the volume of the unit ball in $\R^d$.
Our main result is the following.

\begin{theorem}\label{thm:logfluctuations}
Fix an integer $d \geq 3$.  For each $\gamma$ there exists an $a = a(\gamma,d) < \infty$
such that for all sufficiently large $r$,
\[
\PP \left\{ \B_{r - a \sqrt{\log r}} \subset A(\omega_d r^d) \subset \B_{r+ a \sqrt{\log r}} \right\}^c
\le r^{-\gamma}. \]
\end{theorem}

We treated the case $d=2$ in \cite{JLS} (see also the overview in \cite{JLS0}), where we obtained a similar statement with $\log r$ in place of $\sqrt{\log r}$.  Together with a Borel-Cantelli argument, this in particular implies the following \cite{JLS}:

\begin{corollary} \label{cor:bound}
The maximal distance from $\partial B_r$ to a point in one (but not both) of $\B_r$ and $A(\omega_d r^d)$ is a.s.\ $O(\log r)$ when $d=2$ and $O(\sqrt{\log r})$ when $d > 2$.
\end{corollary}

These results show that internal DLA in dimensions $d\geq 3$ is extremely close to a perfect sphere: when the cluster $A(t)$ has the same size as a ball of radius $r$, its fluctuations around that ball are confined to the $\sqrt{\log r}$ scale (versus $\log r$ in dimension $2$).

In~\cite{JLS} we explained that our method for $d=2$ would also apply in dimensions $d \geq 3$ with the $\log r$ replaced by $\sqrt{\log r}$.  We outlined the changes needed in higher dimensions (stating that the full proof would follow in this paper) and included a key step: Lemma~A, which bounds the probability of ``thin tentacles'' in the internal DLA cluster in all dimensions.  The purpose of this note is to carry out the adaptation of the $d=2$ argument of~\cite{JLS} to higher dimensions.
We remark that in \cite{JLS} we used an estimate from \cite{LBG} to start this iteration, while here we have modified the argument slightly so that this a priori estimate is no longer required.

One way for $A(\omega_d r^d)$ to deviate from the radius $r$ sphere is for it to have a single ``tentacle'' extending beyond the sphere.  The thin tentacle estimate \cite[Lemma~A]{JLS} essentially says that in dimensions $d \geq 3$, the probability that there is a tentacle of length $m$ and volume less than a small constant times $m^d$ (near a given location) is at most $e^{-cm^2}$.  By summing over all locations, one may use this to show that the length of the longest ``thin tentacle'' produced before time $t$ is $O(\sqrt{\log t})$.  To complete the proof of Theorem \ref{thm:logfluctuations}, we will have to show that other types of deviations from the radius $r$ sphere are also unlikely.

Lemma A of \cite{JLS} was also proved for $d=2$, albeit with $e^{-cm^2}$ replaced by $e^{-cm^2/\log m}$.  However, when $d=2$ there appear to be other more ``global'' fluctuations that swamp those produced by individual tentacles. (Indeed, we expect, but did not prove, that the $\log r$ fluctuation bound is tight when $d=2$.)  We bound these other fluctuations in higher dimensions via the same scheme introduced in~\cite{JLS0, JLS}, which involves constructing and estimating certain martingales related to the growth of $A(t)$.  It turns out the quadratic variations of these martingales are, with high probability, of order $\log t$ when $d = 2$ and of constant order when $d \geq 3$, closely paralleling what one obtains for the discrete Gaussian free field (as outlined in more detail in \cite{JLS}).  The connection to the Gaussian free field is made more explicit in \cite{idla-gff}.


Section~\ref{sec:maintheorem} proves Theorem~\ref{thm:logfluctuations}
by iteratively applying higher dimensional analogues of the two main
lemmas of \cite{JLS}.  The lemmas themselves are proved in
Section~\ref{sec:earlylate}, which is the heart of the argument.
Section \ref{sec:green} contains preliminary estimates about random walks
that are used in Section~\ref{sec:earlylate}.

\subsection*{A brief history of internal DLA fluctuation bounds} The history of fluctuation bounds such as the
one in Corollary~\ref{cor:bound} is as follows.  In 1991, Lawler,
Bramson, and Griffeath proved that the limit shape of internal DLA
from a point is the ball in all dimensions \cite{LBG}.  In 1995 Lawler
gave a more quantitative proof, showing that the fluctuations of $A(\omega_d r^d)$ from the
ball of radius~$r$ are at most of order $O(r^{1/3} \log^4 r)$ \cite{Lawler95}.
In December 2009, the present authors announced the bound $O(\log r)$ on fluctuations in dimension $d=2$ \cite{JLS0} and gave an overview of the argument, making clear that the details remained to be written.
In April 2010, Asselah and Gaudilli\`ere \cite{AGa}
gave a proof, using different methods from \cite{JLS0}, of the bound
$O(r^{1/(d+1)})$ in all dimensions, improving the Lawler bound for all
$d\ge 3$.  In September 2010, Asselah and Gaudilli\`ere improved this to
$O((\log r)^2)$ in all dimensions $d \geq 2$ with an $O(\log r)$ bound on ``inner'' errors \cite{AGb}.
In October 2010 the present authors proved the $O(\log r)$ bounds (announced in December 2009) for
dimension $d=2$
and outlined the proof of the $O(\sqrt{\log r})$ bound
for dimensions $d\ge 3$ \cite{JLS}.  In November 2010, Asselah and Gaudilli\`ere
gave a second proof of the $O(\sqrt{\log r})$ bound \cite{AGc}.  Their proof uses
methods from \cite{AGb} along with Lemma~A of \cite{JLS}
to bound ``outer'' errors and a new large deviation bound (in some
sense symmetric to Lemma~A) to bound ``inner'' errors.

More references and a more general discussion of internal DLA history appear in \cite{JLS}.

\section{Proof of Theorem \ref{thm:logfluctuations}}
\label{sec:maintheorem}

Let $m$ and $\ell$ be positive real numbers.  We say that $x \in \Z^d$ is $m$-early if \[ x \in A(\omega_d (|x|-m)^d), \]
where $\omega_d$ is the volume of the unit ball in $\R^d$.  Likewise, we say that $x$ is $\ell$-late if  \[ x \notin A(\omega_d (|x|+\ell)^d). \]
Let $\Early_m[T]$ be the event that some point of $A(T)$ is $m$-early.  Let $\Late_\ell[T]$ be the event that some point of $\B_{(T/\omega_d)^{1/d}-\ell}$ is $\ell$-late.  These events correspond to ``outer'' and ``inner'' deviations of $A(T)$ from circularity.

\begin{lemma}\label{earlyimplieslate}
{\em (Early points imply late points)} Fix a dimension $d \geq 3$.  For each $\gamma \geq 1$, there is a constant $C_0=C_0(\gamma,d)$, such that for all sufficiently large $T$, if~$m \ge C_0 \sqrt{\log T}$ and $\ell \le m/C_0$, then
\[
\PP(\Early_m[T] \cap \Late_\ell [T]^c) < T^{-10\gamma}.
\]
\end{lemma}

\begin{lemma}\label{lateimpliesearly}
{\em (Late points imply early points)}
Fix a dimension $d \geq 3$.  For each $\gamma \geq 1$, there is a constant $C_1=C_1(\gamma, d)$ such that for all sufficiently large~$T$, if $m \ge \ell \ge C_1 \sqrt{\log T}$ and $\ell \ge C_1((\log T)m)^{1/3}$, then
\[
\PP(\Early_m[T]^c \cap \Late_\ell [T]) \le T^{-10\gamma}.
\]
\end{lemma}

\begin {figure}[H]
\begin {center}
\includegraphics [width=5in]{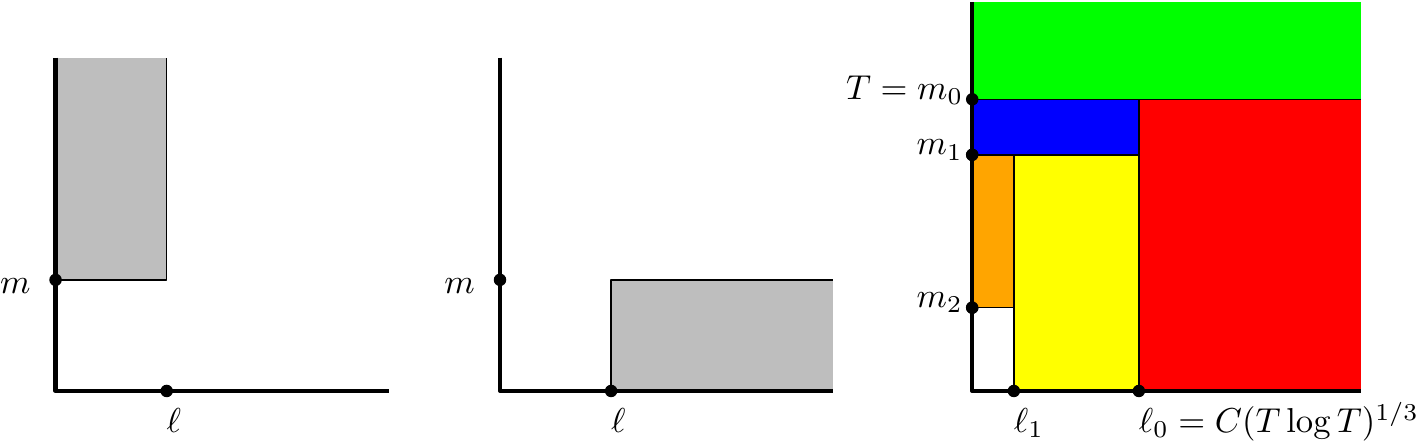}
\caption {\label{mlrectangles}  Let $m^T$ be the smallest $m'$ for which $A(T)$ contains an $m'$ early point.   Let $l^T$ be the largest $\ell'$ for which some point of $B_{(T/\omega_d)^{1/d}-\ell'}$ is $\ell'$-late. By Lemma \ref{earlyimplieslate}, $(\ell^T, m^T)$ is unlikely to belong to the semi-infinite rectangle in the left figure if $\ell < m/C_0$.  By Lemma \ref{lateimpliesearly}, $(\ell^T,m^T)$ is unlikely to belong to the semi-infinite rectangle in the second figure if $\ell \geq C_1((\log T)m)^{1/3}$.  Theorem~\ref{thm:logfluctuations} will follow because $m^T> m_0 = T$ is impossible and the other rectangles on the right are all (by Lemmas~\ref{earlyimplieslate} and~\ref{lateimpliesearly}) unlikely.}
\end {center}
\end {figure}

We now proceed to derive Theorem~\ref{thm:logfluctuations} from Lemmas~\ref{earlyimplieslate} and~\ref{lateimpliesearly}.  The lemmas themselves will be proved in Section \ref{sec:earlylate}.  Let $C=\max(C_0,C_1)$.
We start with
\[
m_0 = T.
\]
Note that $A(T) \subset \B_T$, so $\PP(\Early_{T}[T]) = 0$.  Next, for $j \geq 0$ we let
\[
\ell_j = \max(C((\log T)m_j)^{1/3}, C\sqrt{\log T})
\]
and
\[
m_{j+1} = C \ell_j.
\]
By induction on $j$, we find
\begin{align*}
&\PP(\Early_{m_j}[T]) < 2j T^{-10\gamma} \\
&\PP(\Late_{\ell_j}[T])  < (2j+1)T^{-10\gamma}.
\end{align*}

To estimate the size of $\ell_j$, let $K=C^4 \log T$ and note that $\ell_j \leq \ell'_j$, where
\[
\ell'_0  = (KT)^{1/3}; \quad \ell'_{j+1} = \max( (K\ell'_j)^{1/3}, K^{1/2}).
\]
Then
\[
\ell'_j \le \max (K^{1/3 + 1/9 + \cdots + 1/3^j} T^{1/3^j}, K^{1/2})
\]
so choosing $J= \log T$ we have
\[
T^{1/3^J}  < 2
\]
and
\[
\ell_J \le 2K^{1/2} \le C \sqrt{\log T}.
\]

The probability that $A(T)$ has $\ell_J$-late points or $m_J$-early points is at most
\[
(4J+1) T^{-10\gamma} < T^{-9\gamma} < r^{-\gamma}.
\]

Setting $T=\omega_d r^d$, $\ell = \ell_J$ and $m=m_J$, we conclude that if $a$ is sufficiently large, then
\[
\PP \left\{ \B_{r - a \sqrt{\log r}} \subset A(\omega_d r^d) \subset \B_{r+ a \sqrt{\log r}} \right\} \leq
\PP ( \Early_{m}[T] \cup \Late_{\ell}[T] ) < r^{-\gamma}
\]
which completes the proof of Theorem~\ref{thm:logfluctuations}.

\section{Green function estimates on the grid}
\label{sec:green}



This section assembles several Green function estimates that we need to prove Lemmas \ref{earlyimplieslate} and \ref{lateimpliesearly}.  The reader who prefers to proceed to the heart of the argument may skip this section on a first read and refer to the lemma statements as necessary.  Fix $d \geq 3$ and consider the $d$-dimensional grid
	\[ \Grid = \{(x_1,\ldots,x_d) \in \R^d \setcolon \mbox{at most one } x_i \notin \Z \}. \]
In many of the estimates below, we will assume that a positive integer $k$ and a $y\in \Z^d$ have been fixed.  We write $s=|y|$ and
\[
\O = \O(y,k) := \Grid\cap B_{s+k}\backslash \{y\}.
\]
For $x \in \partial \Omega$, let $$P(x) =P_{y,k}(x)$$ be the probability that a Brownian motion on the grid $\Grid$ (defined in the obvious way; see \cite{JLS})
starting at $x$ reaches $y$ before exiting $B_{s+k}$.  Note that $P$ is {\bf grid harmonic} in $\Omega$ (i.e., $P$ is linear on each segment of $\Omega \setminus \Z^d$, and for each $x \in \Omega \cap \Z^d$, the sum of the slopes of $P$ on the $2d$ directed edge segments starting at $x$ is zero).  Boundary conditions are given by $P(y)=1$ and $P(x)=0$ for $x \in (\partial \Omega) \setminus \{y \}$.  The point $y$ plays the role that $\zeta$ played in \cite{JLS}, and $P$ plays the role of the discrete harmonic function $H_\zeta$.  One difference from \cite{JLS} is that we will take $y$ inside the ball (i.e., $k \geq 1$) instead of on the boundary.

To estimate $P$ we use the discrete Green function $g(x)$, defined as the expected number of visits to $x$ by a simple random walk started at the origin in $\Z^d$.  The well-known asymptotic estimate for $g$ is~\cite{Uch}
\begin{equation}
\label{eq:uchiyama}
\left|g(x) - a_d|x|^{2-d}\right|  \le C |x|^{-d}
\end{equation}
for dimensional constants $a_d$ and $C$ (i.e., constants depending only on the dimension $d$).
We extend $g$ to a function, also denoted $g$, defined on
the grid $\Grid$ by making $g$ linear on each segment between
lattice points. Note that $g$ is grid harmonic on $\Grid \setminus \{0 \}$.

Throughout we use $C$ to denote a large positive dimensional constant, and $c$ to denote a small positive dimensional constant, whose values may change from line to line.

\begin{lemma}\label{upperbound} There is a dimensional constant $C$ such that
\begin{itemize}
\item[\em (a)] $\disp P(x) \le {C}/{(1 + |x-y|^{d-2})}$.
\item[\em (b)] $\disp P(x) \le {Ck(s+k+1-|x|)}/{|x-y|^{d}}, \qquad \mbox{for } |x-y| \ge k/2$.
\item[\em (c)] $\disp \max_{x\in \B_r} P(x) \le {Ck}/{(s-r-k)^{d-1}}$ for $r < s-2k$.
\end{itemize}
\end{lemma}

\begin{proof}
The maximum
principle (for grid harmonic functions) implies
$Cg(x-y)\ge P(x)$ on $\O$, which gives part (a).

The maximum principle also implies that for $x\in \O$,
\begin{equation}
\label{eq:greendifference}
P(x) \le C(g(x-y) - g(x-y^*))
\end{equation}
where $y^*$ is the one of the lattice points nearest to $(s + 2k + C_1)y/s$.
Indeed, both sides are grid harmonic on $\O$, and the right side is positive on $\partial B_{s+k}$ by \eqref{eq:uchiyama}, so it suffices to take $C=(g(0) - g(y-y^*))^{-1}$.

Combining (\ref{eq:uchiyama}) and (\ref{eq:greendifference}) yields the bound
\[
P(x) \le \frac{Ck}{|x-y|^{d-1}},  \quad \mbox{for} \ |x-y|\ge 2k.
\]
Next, let $z\in \partial B_{s+k}$ be such that $|z-y| = 2L$, with $L \ge 2k$.
The bound above implies
\[
P(x) \le \frac{Ck}{L^{d-1}},  \quad \mbox{for} \ x\in B_L(z)
\]
Let $z^*$ be one of the lattice points nearest to $(s+k + L+ C_1)z/|z|$.  Then
\[
F(x) = a_d L^{2-d} - g(x-z^*)
\]
is comparable to $L^{2-d}$ on $\partial B_{2L}(z^*)$ and positive
outside the ball $B_L(z^*)$ (for a large enough dimensional
constant $C_1$ --- in fact, we can also do this with $C_1=1$ with
$L$ large enough).  It follows that
\[
P(x) \le C(k/L^{d-1})(L^{d-2})F(x)
\]
on $\partial (B_{2L}(z^*) \cap \O)$ and hence by the maximum
principle on $B_{2L}(z^*) \cap \O$.  Moreover,
\[
F(x) \le C(s+k + 1 - |x|)/L^{d-1}
\]
for $x$ a multiple of $z$ and $s + k - L \le |x| \le s+k$.
Thus for these values of $x$,
\[
P(x) \le C(k/L)F(x) \le Ck(s+k+1-|x|)/L^d
\]
We have just confirmed the bound of part (b) for points $x$
collinear with $0$ and $z$, but $z$ was essentially arbitrary.
To cover the cases $|x-y| \le 2k$ one has to use exterior
tangent balls of radius, say $k/2$, but actually the
upper bound in part (a) will suffice for us in the range $|x-y| \le Ck$.

Part (c) of the lemma follows from part (b).
\end{proof}

The mean value property (as typically stated for continuum harmonic functions) holds only approximately for discrete harmonic functions.  There are two choices for where to put the approximation: one can show that the average of a discrete harmonic function $u$ over the discrete ball $\B_r$ is approximately $u(0)$, or one can find an approximation $w_r$ to the discrete ball $\B_r$ such that averaging $u$ with respect to $w_r$ yields \emph{exactly} $u(0)$.  The divisible sandpile model of \cite{LP09} accomplishes the latter.  In particular, the following discrete mean value property follows from Theorem 1.3 of \cite{LP09}.

\begin{lemma} {\em (Exact mean value property on an approximate ball)}
\label{lem:sandpile}
For each real number $r > 0$, there is a function $w_r : \Z^d \to [0,1]$ such that
	\begin{itemize}
	\item $w_r(x) = 1$ for all $x\in \B_{r-c}$, for a constant $c$ depending only on $d$.
	\item $w_r(x) = 0$ for all $x\notin \B_{r}$.
	\item	 For any function $u$ that is discrete harmonic on $\B_{r}$,
\[
\sum_{x\in \Z^d} w_r(x)(u(x)- u(0)) = 0.
\]
	\end{itemize}
\end{lemma}

The next lemma bounds sums of $P$ over discrete spherical shells and discrete balls.  Recall that $s=|y|$.

\begin{lemma}\label{meanvalue} There is a dimensional constant $C$ such that
\begin{itemize}
\item[\em (a)] $\disp \sum_{x \in \B_{r+1}\backslash \B_r} P(x) \le Ck$ for all $r \le s+k$.
\item[\em (b)] $\disp \left|\sum_{x \in \B_r} (P(x) -P(0)) \right|\le Ck$ for all $r \le s$. 
\item[\em (c)] $\disp \left|\sum_{x \in \B_{s+k}} (P(x) -P(0)) \right|\le Ck^2$.
\end{itemize}
\end{lemma}

\begin{proof}
Part (a) follows from Lemma \ref{upperbound}: Take the worst shell, when $r=s$.
Then the lattice points with $|x-y| \le k$, $s \le |x|\le s+1$
are bounded by Lemma \ref{upperbound}(a)
\[
\int_0^k s^{2-d} s^{d-2}ds = k
\]
(volume element on disk with thickness $1$ and radius $k$ in
$\Z^{d-1}$ is $s^{d-2}ds$.)  For the remaining portion of the shell,
Lemma \ref{upperbound}(b) has numerator  $k(s+k - s) = k^2$,
so that
\[
\int_{k}^\infty k^2 s^{-d} s^{d-2}ds =  k
\]
Next, for part (b), let $w_r$ be as in Lemma~\ref{lem:sandpile}.  Since $P$ is discrete harmonic in $\B_s$, we have for $r \leq s$
	\begin{align*}  \sum_{x \in \Z^d} w_r(x)(P(x) - P(0)) = 0.
	\end{align*}
Since $w_r$ equals the indicator $\one_{\B_r}$ except on the annulus $\B_r \setminus \B_{r-c}$, and $|w_r| \leq 1$, we obtain
\begin{align*}
\left|\sum_{x\in \B_r} (P(x) -P(0)) \right|
&\le
\sum_{x\in \B_{r}\setminus \B_{r-c}} |w_r(x)| \left| P(x) -P(0) \right|  \\
&\le
\sum_{x\in \B_{r}\setminus \B_{r-c}} (P(x) + P(0))  \\
&  \le Ck.
\end{align*}
In the last step we have used part (a) to bound the first term; the second term is bounded by Lemma \ref{upperbound}(b), which says that $P(0) \le Ck/s^{d-1}$.

Part (c) follows by splitting the sum over $\B_{s+k}$ into $k$ sums over spherical shells $\B_{s+j} \setminus \B_{s+j-1}$ for $j=1,\ldots,k$, each bounded by part (a), plus a sum over the ball $\B_s$, bounded by part (b).
\end{proof}

Fix $\alpha>0$, and consider the level set
	\[ U  = \{x \in \Grid \mid g(x)>\alpha\}. \]
For $x \in \partial U$, let $p(x)$ be the probability that a Brownian motion started at the origin in $\Grid$ first exits $U$ at $x$.

\begin{lemma}\label{Greensformula}
Choose $\alpha$ so that $\partial U$ does not intersect~$\Z^d$. For each $x \in \partial U$, the quantity $p(x)$ equals the directional derivative of $g/2d$ along the directed edge in $U$ starting at $x$.
\end{lemma}

\begin{proof}
We use a discrete form of the divergence theorem
\begin{equation}\label{div}
\int_U \div V = \sum_{\partial U} \nu_U \cdot V.
\end{equation}
where $V$ is a vector-valued function on the grid, and the integral on the left is a one-dimensional integral
over the grid.  The dot product $\nu_U \cdot V$ is defined as $e_j \cdot V(x-0e_j)$, where $e_j$ is the unit vector pointing toward $x$ along the unique incident edge in~$U$.
To define the divergence, for $z = x + te_j$, where $0 \le t < 1$ and $x \in \Z^d$, let
\[
\div V (z) := \frac{\de}{\de x_j} e_j\cdot V(z)
+ \delta_x(z)\sum_{j=1}^d (e_j\cdot V(x+ 0e_j) -e_j\cdot V(x-0e_j)).
\]

If $f$ is a continuous function on $U$ that is $C^1$ on each
connected component of $U-\Z^d$, then the gradient of $f$ is the vector-valued function
\[
V = \nabla f= (\de f/\de x_1, \de f/\de x_2, \dots , \de f/\de x_d)
\]
with the convention that the entry $\de f/\de x_j$ is $0$ if the segment
is not pointing in the direction $x_j$.  Note that $\nabla f$ may be discontinuous at points of $\Z^d$.

\old{
If $x$ is a lattice boundary rather than an interior point, then
there are no delta functions at $x$, which is outside the
region.  This is consistent with the definition being
\[
\div V = \sum_j \frac{\de}{\de x_j} e_j\cdot V(z)
\]
in the distribution sense when tested against smooth
functions supported on the interior.

Finally,  at a lattice site $x$ on the boundary of a region
$\O$ we define the dot product with the outward pointing normal as
\[
\nu_U \cdot V(x) = \sum_{j=1}^d (-e_j\cdot V(x+ 0e_j) +e_j\cdot V(x-0e_j))
\]
where the terms of the sum that appear are the ones that
can be evaluated from inside $U$, i.~e. for the term $-e_j\cdot V(x+0e_j)$
to appear we must have $x+te_j\in U$ for sufficiently small $t>0$.
Note that all $2d$ terms might appear if $x$ is an isolated point of the
complement of $U$ in the grid.   If $x$ is a boundary point of $U$,
but not a lattice site, then there is only one, unambiguous term,
the continuous limit (from the inside) of the value of $e_j\cdot V$
along the segment in the $x_j$ direction.
}

Let $G = -g/2d$, so that $\div \nabla G  = \delta_0$.  If $u$ is
grid harmonic on $U$, then $\div \nabla u = 0$ and
\[
\div(u\nabla G - G\nabla u) = u(0) \delta_0.
\]
Indeed, on each segment this is the same as $(uG' - u'G)'  = u'G' - u'G' + uG'' - u''G = 0$
because $u $ and $G$ are linear on segments.   At lattice points
$u$ and $G$ are continuous, so the divergence operation commutes
with  the factors $u$ and $G$ and gives exactly one nonzero delta term,
the one indicated.

Let $u(y)$ be the probability that Brownian motion on $U$ started at $y$ first exits $U$ at $x$.  Since $u$ is grid-harmonic on $U$, we have $\div \nabla u = 0$ on $U$, hence by the divergence theorem
\[
u(0) = \int_{U} \div(u\nabla G - G\nabla u) = \sum_{\partial U} u\, \nu_U \cdot \nabla G. \qed
\]
\renewcommand{\qedsymbol}{}
\end{proof}

Next we establish some lower bounds for $P$.

\begin{lemma}\label{lowerbound} There is a dimensional constant $c>0$ such that
\begin{itemize}
\item[(a)] $\disp P(0) \ge ck/s^{d-1}$.
\item[(b)] Let $k=1$, and $z = (1-\frac{2m}{s})y$.  Then
\[
\min_{x\in \B(z,m)} P(x)  \ge c/m^{d-1}.
\]
\end{itemize}
\end{lemma}

\begin{proof}
By the maximum principle,
there is a dimensional constant $c>0$ such that
\[
P(x) \ge c (g(x-y) - a_d (k/2)^{2-d})
\]
for $x\in B_{k/2}(y)$.  In particular,
\[
P(x) \ge c k^{2-d} \quad \mbox{for all } \ |x-y| \le k/4
\]
Now consider the region
\[
U = \{x\in \Grid: g(x) > a_d (s')^{2-d}\}
\]
where $s'$ is chosen so that $|s' - (s-k/8)| < 1/2$ and all of the
boundary points of $U$ are non-lattice points.  (A generic value of $s'$ in the given range will suffice.)

By \eqref{eq:uchiyama}, this set is within unit distance of
the ball of radius $s-k/8$.  Let $p(z)$ represent the
probability that a Brownian motion on the grid starting from
the origin first exits $U$ at $z\in \partial U$.  Thus
\begin{equation}\label{Kakutani}
u(0) = \sum_{z\in \partial U} u(z) p(z)
\end{equation}
for all grid harmonic functions $u$ in $U$.

Take any boundary point of $z\in \partial U$.  Take the nearest lattice
point~$z^*$.  Let $z_j$ be a coordinate of $z$ largest in absolute value.  Then
$|z_j| \ge |z|/d$. The rate of change of $|x|^{2-d}$ in the $j$th
direction near $z$ has size $\ge 1/d|z|^{d-1}$, which is much larger
than the error term $C|z|^{-d}$ in \eqref{eq:uchiyama}.  It follows that on
the segment in that direction, where
the function $g(x) - a_d (s-k/8)^{2-d}$ changes sign, its
derivative is bounded below by $1/2d|z|^{d-1}$.  In other
words, by Lemma \ref{Greensformula}, within distance~$2$ of every boundary point of $z\in \partial U$
there is a point $z'\in \partial U$ for which  $p(z')\ge c/s^{d-1}$.
There are at least $ck^{d-1}$ such points in the ball $\B_{k/4}(y)$
where the lower bound for $P$ was $ck^{2-d}$, so
\[
P(0) \ge  ck^{2-d} k^{d-1} /s^{d-1} = ck/s^{d-1}.
\]

Next, the argument for Lemma \ref{lowerbound}(b) is nearly the same.
We are only interested in $k=1$.  It is obvious that for
points $x$ within constant distance of $y$ (and unit distance
from the boundary at radius $s+1$, the values of $P(x)$ are
bounded below by a positive constant.  We then bound
$P((s-2m)y/|y|)$ from below using the same argument as above,
but with Green's function for a ball of radius comparable to $m$.
Finally, Harnack's inequality says that the values of $P(x)$
for $x$ in the whole ball of size $m$ around this point
$(s-2m)y/|y|$ are comparable.
\end{proof}

\section{Proofs of main lemmas}
\label{sec:earlylate}
The proofs in this section make use of the martingale
	\[ M(t) = M_{y,k}(t) := \sum_{x \in A_{y,k}(t)} (P(x)-P(0)) \]
where $A_{y,k}(t)$ is the modified internal DLA cluster in which particles are stopped if they exit~$\Omega$.  As in \cite{JLS}, we view $A_{y,k}(t)$ as a multiset: points on the boundary of~$\Omega$ where many stopped particles accumulate are counted with multiplicity in the sum defining $M$.  In addition to these stopped particles, the set $A_{y,k}(t)$ contains one more point, the location of the currently active particle performing Brownian motion on the grid $\Grid$.

Recall that $P=P_{y,k}$ and $M=M_{y,k}$ depend on $k$, which is the distance from $y$ to the boundary of~$\Omega$.  We will choose $k=1$ for the proof of Lemma~\ref{earlyimplieslate}, and $k=a\ell$ for a small constant~$a$ in the proof of Lemma~\ref{lateimpliesearly}.  Taking $k>1$ is one of the main differences from the argument in \cite{JLS}.

\begin{proof}[Proof of Lemma \ref{earlyimplieslate}]
The proof follows the same method as \cite[Lemma 12]{JLS}. We highlight here the changes needed in dimensions $d \geq 3$.
We use the discrete harmonic function $P(x)$ with $k=1$.  Fix $z \in \Z^d$, let $r=|z|$ and $y = (r+2m)z/r$.
Let
\[
T_1 = \ceil{\omega_d (r-m)^d}
\]
where $\omega_d$ is the volume of the unit ball in $\R^d$.  If $z$ is $m$-early, then $z\in A(T_1)$; in particuar, this means that $r \ge m$, so that $r+m$, $r+2m$ are all comparable to $r$.
Since $k=1$, we have by Lemmas~\ref{upperbound}(c) and~\ref{lowerbound}(a)
\[
P(0) \approx 1/r^{d-1},
\]
where $\approx$ denotes equivalence up to a constant factor depending only on~$d$.

First we control the quadratic variation
	\[ S(t) = \lim_{\substack{0 = t_0 \leq \ldots \leq t_N = t \\ \max (t_i - t_{i-1}) \to 0}} \; \sum_{i=1}^{N} (M(t_{i})-M(t_{i-1}))^2  \]
on the event $\Early_{m+1}[T]^c$ that there are no $(m+1)$-early points by time $T$.  As in \cite[Lemma~9]{JLS}, there are independent standard Brownian motions $\widetilde{B}^0, \widetilde{B}^1, \ldots$ such that each increment $(S(n+1) - S(n))\one_{\Early_{m+1}[T]^c}$ is bounded above by the first exit time of $\widetilde{B}^n$ from the interval $[-a_n,b_n]$, where
	\begin{align*} a_n &= P(0) \approx \frac{1}{r^{d-1}} \\  b_n &= \max_{|x| \leq (n/\omega_d)^{1/d} + m+1} P(x) \leq \frac{1}{[r+2m - ((n/\omega_d)^{1/d} + m +1)]^{d-1}}.
	\end{align*}
Here we have used Lemma~\ref{upperbound}(b) in the bound on $b_n$.

Unlike in dimension $2$, we will use the large deviation bound for Brownian exit times \cite[Lemma~5]{JLS} with $\lambda=cm^2$ instead of $\lambda=1$.  Here $c$ is a constant depending only on $d$.  Note that $b_n \leq 1/m^{d-1}$, for all $n\leq T_1$, so this is a valid choice of $\lambda$ in all dimensions $d \geq 3$ (that is, the hypothesis
$\sqrt{\lambda}(a_n + b_n)  \leq 3$ of \cite[Lemma~5]{JLS} holds).  We obtain
\begin{align*}
\log \EE \left[ e^{\lambda S(T_1)}1_{\Early_{m+1}[T]^c} \right]
& \le \sum_{n=1}^{T_1} 10 \lambda a_n b_n \\
& \leq \int_{1}^{T_1} \lambda \frac{C}{r^{d-1}}\frac{1}{(r+m - (n/\omega_d)^{1/d} - 1)^{d-1}} dn \\
& \leq \int_{1}^r
\lambda \frac{C}{r^{d-1}} \frac{1}{(r+m - j- 1)^{d-1}} j^{d-1} dj  \\
& \le \int_{1}^r
\frac{C \lambda \, dj}{(r+m - j - 1)^{d-1}}   \le C \lambda/m^{d-2 }.
\end{align*}
Note that the last step uses $d\ge 3$.   Taking $\lambda=cm^2$ for small enough $c$ we obtain
\[
\EE \left[ e^{cm^2 S(T_1)} 1_{\Early_{m+1}[T]^c} \right] \leq e^{m^2/m^{d-2}} \le e^m.
\]
Therefore, by Markov's inequality,
\begin{equation}
\label{eq:firstterm}
\PP( \{ S(T_1) > 1/c \} \cap \Early_{m+1}[T]^c) \le  e^{m-m^2} < T^{-20\gamma}.
\end{equation}

Fix $z \in \B_T$ and $t \in \{1,\ldots,T\}$, and let $Q_{z,t}$ be the event that $z \in A(t) \setminus A(t-1)$ and $z$ is $m$-early and no point of $A(t-1)$ is $m$-early.  This event is empty unless $(t/\omega_d)^{1/d} + m \leq |z| \leq (t/\omega_d)^{1/d} + m + 1$; in particular, the first inequality implies $t \leq T_1$.
We will bound from below the martingale $M(t)$ on the event $Q_{z,t} \cap \Late_\ell[T]^c$.  With no $\ell$-late point, the ball $\B_{r-m-\ell-1}$ is entirely filled by time $t$.  Lemma \ref{meanvalue}(b)
shows that the sites in this ball contribute at most a constant to $M(t)$ (recall that $k=1$).
The thin tentacle estimate \cite[Lemma~A]{JLS} says that except for an event of probability $e^{-cm^2}$,
there are order $m^d$ sites in $A(t)$ within the ball $\B(z,m)$.  By Lemma \ref{lowerbound}(b), $P$ is bounded below by $c/m^{d-1}$ on this ball, so
these sites taken together contribute order $m$ to $M(t)$.  Each of the
remaining terms in the sum defining $M(t)$ is bounded below by $-P(0)$, and there are at most $\ell r^{d-1}$ sites in $A(t) \setminus \B_{r-m-\ell-1}$.  So these terms contribute at least
\[
-\ell r^{d-1} (1/r^{d-1}) = -\ell  \ge - m/C
\]
which cannot overcome the order $m$ term.  Thus
	\begin{align}
	\label{eq:secondterm}
	 \PP (Q_{z ,t } \cap \{M_\zeta(t ) < m/C \} \cap \Late_\ell[t ]^c) < e^{-cm^2}.
	 \end{align}
We conclude that
\begin{align*}
\PP(Q_{z ,t} \cap \Late_\ell[T]^c)
&\le \PP (Q_{z ,t }\cap \{S(t ) > 1/c \}) \\
& \qquad+ \PP (Q_{z ,t } \cap \{M(t ) < m/C \} \cap \Late_\ell[t ]^c) \\
& \qquad \qquad + \PP ( \{S(t ) \le 1/c \} \cap \{ M(t ) \ge m/C \}).
\end{align*}
The first two terms are bounded by \eqref{eq:firstterm} and \eqref{eq:secondterm}.  Since $M(t) = B(S(t))$ for a standard Brownian motion~$B$, the final term is bounded by
	\[ \PP \left\{ \sup_{0 \leq s \leq 1/c} B(s) \geq m/C \right\} < e^{-c(m/C)^2/2} < T^{-20\gamma}. \qed \]
\renewcommand{\qedsymbol}{}
\end{proof}

\begin{proof}[Proof of Lemma \ref{lateimpliesearly}]
Fix $y \in \Z^d$, and let $L[y]$ be the event that $y$ is $\ell$-late.  Let $s = |y|$, and set $k=a\ell$ in the definition of $P$.  Here $a>0$ is a small dimensional constant chosen below.  Note that the hypotheses on $m$ and $\ell$ imply that $\ell$ is at least of order $\sqrt{\log T}$; after choosing~$a$, we take the constant $C_1$ appearing in the statement of the lemma large enough so that $k^2 > 1000 \gamma \log T$.

\medskip
Case 1.  $1 \le s \le 2k$.  Then $P(0) \approx 1/s^{d-2}$.  Let
\[
T_1 = \floor{\omega_d (s + \ell)^d}
\]
With $a_n = P(0)$ and $b_n=1$, we have $S(n+1)-S(n) \leq \tau_n$, where $\tau_n$ is the first exit time of the Brownian motion $\widetilde{B}^n$ from the interval $[-a_n,b_n]$. (Note that because we take $b_n=1$, the indicator $\one_{\Early_{m+1}[T]^c}$ is not needed here as it was in the proof of Lemma~\ref{earlyimplieslate}.) We obtain
\[
\log\EE e^{S(T_1)} \leq \sum_{t=1}^{T_1} \log \EE e^{\tau_n} \leq T_1 P(0).
\]
Let $Q = T_1 P(0)$.  By Markov's inequality, $\PP(S(T_1) > 2Q) \leq e^{-Q}$.

On the event $L[y]$, the site $y$ is still not occupied at time $T_1$.  Accordingly, the largest $M(T_1)$ can be is if $A_{y,k}(T_1)$ fills the whole ball $\B_{s+k}$ (except for $y$), and then the rest of the particles will have to collect on the boundary where $P$ is zero.  The contribution from $\B_{s+k}$ is at most $Ck^2$ by Lemma~\ref{meanvalue}(c).  The number of particles stopped on the boundary is at least
\[
T_1 - 2\omega_d(s+k)^d  \ge \frac{T_1}{2}.
\]
Therefore, on the event $L[y]$ we have
\begin{equation}
\label{eq:verynegative}
M (T_1) \le  Ck^2 - \frac{T_1}{2}P(0).
\end{equation}
Note that $Q: = T_1 P(0)  \approx (s+\ell)^d / s^{d-2} \geq \ell^d / (k/2)^{d-2}$,
so by taking~$a=k/\ell$ sufficiently small, we can ensure that the right side of \eqref{eq:verynegative} is at most $-Q/4$.  Also, $Q \geq \ell^2 \geq 1000 \gamma \log T$.  Since $M(t) = B(S(t))$ for a standard Brownian motion~$B$, we conclude that
\begin{align*}
\PP(L[y]) &\leq \PP(S(T_1)>2Q) + \PP \left\{ \inf_{0 \leq s \leq 2Q} B(s) \leq -Q/4 \right\} \\
	&\leq e^{-Q} + e^{-(Q/4)^2/4Q} \\
	& < T^{-20\gamma}.
\end{align*}

\bigskip

Case 2.  $s \geq 2k$.  Then by Lemma~\ref{upperbound}(c) with $r=1$, and Lemma~\ref{lowerbound}(a), we have $P(0) \approx k/s^{d-1}$.  First take
\[
T_0 = \floor{\omega_d(s+k -3m)^d}
\]
(or $T_0=0$ if $s+k-3m \le 0$).  As in the previous lemma (but taking $\lambda=1$ instead of $\lambda = cm^2$) we have
\[
\log \EE \left[ e^{S(T_0)} 1_{\Early_m[T]^c} \right] \leq C \frac{k}{s^{d-1}} \int_0^{T_0} \frac{d n}{ \bigl( s+k-(n/\omega_d)^{1/d}\bigr)^{d-1}} \leq C k/m^{d-2} \leq C.
\]
The last inequality follows from $d \geq 3$ and $m \geq k/a$.  By Markov's inequality,
\[
\PP ( \{S(T_0) > C + k^2\} \cap \Early_m[T]^c ) <  e^{-k^2} < T^{-20\gamma}.
\]
Now since
\[
(T_1 - T_0) P(0) \approx m s^{d-1} (k/s^{d-1}) = km
\]
we have
\[
\log\EE e^{S(T_1)-S(T_0)} \le Ckm.
\]
Thus (since $km \ge k^2$)
\begin{equation}
\label{eq:shorttime}
\PP ( \{ S(T_1) > 2Ckm \} \cap \Early_m[T]^c) < 2 T^{-20\gamma}.
\end{equation}

As in case 1, the martingale $M(T_1)$ is largest if the ball $\B_{s+k}$ is completely filled, and in that case the total contribution of sites in this ball is at most $Ck^2$.  On the event $L[y]$, the number of particles stopped on the boundary of $\Omega$ at time $T_1$ is at least
	\[ T_1 - \# \B_{s+k} \geq \omega_d ((s+\ell)^d - (s+k+C)^d) \approx \ell s^{d-1}. \]
Each such particle contributes $-P(0) \approx -k/s^{d-1}$ to $M(T_1)$, for a total contribution of order $-k\ell = -k^2 / a$.  Taking $a$ sufficiently small we obtain $M(T_1) \leq Ck^2 - k^2 / a \leq - k^2$.
We conclude that
	\begin{align*} \PP( L[y] \cap \Early_m[T]^c ) &\leq \PP( \{ S(T_1) > 2Ckm \} \cap \Early_m[T]^c) + \\
	&\qquad + \PP( \{ S(T_1) \leq 2Ckm \} \cap \{ M(T_1) \leq - k^2 \} ).
	\end{align*}
The first term is bounded above by \eqref{eq:shorttime}, and the second term is bounded above by
	\[ \PP \left\{ \inf_{s \leq 2Ckm} B(s) \leq -k^2 \right\} \leq e^{-k^4 / 4Ckm} < T^{-20\gamma}. \]
Hence $ \PP( L[y] \cap \Early_m[T]^c ) < 3T^{-20\gamma}$.
Since $\Late_\ell[T]$ is the union of the events $L[y]$ for $y \in \mathcal{B} := \B_{(T/\omega_d)^{1/d}-\ell}$, summing over $y \in \mathcal{B}$ completes the proof.
\end{proof}

\end{document}